\documentclass[compress, preprint,9pt]{elsarticle}

\usepackage[abs]{overpic}
\usepackage{graphicx}
\usepackage{subfigure}
\usepackage{stmaryrd}
\usepackage{amsfonts, color, amssymb}
\usepackage{amsmath}
\usepackage{amsthm}
\usepackage{epsfig}
\usepackage{psfrag}
\usepackage{bm}
\usepackage{paralist}
\usepackage{algorithm}
\usepackage{algorithmicx}
\usepackage{algpseudocode}
\usepackage{appendix}
\usepackage{caption}
\usepackage{hyperref}
\usepackage{color}

\usepackage[misc]{ifsym}

\usepackage[a4paper, total={6in,8in}]{geometry}

\allowdisplaybreaks[4]
\newtheorem{theorem}{Theorem}[section] %
\newtheorem{remark}{Remark}[section]
\newtheorem{example}{Example}[section]

\newtheorem{lemma}{Lemma}[section]

\newtheorem{assumption}{Assumption}[section] %
\def\3bar{{|\hspace{-.02in}|\hspace{-.02in}|}}

\allowdisplaybreaks[4]

\allowdisplaybreaks 

\numberwithin{equation}{section}

\journal{Journal of Computational and Applied Mathematics}

\begin{document}

\begin{frontmatter}

\title{Convergence analysis of a weak Galerkin finite element method on a Shishkin mesh for a singularly perturbed fourth-order problem in 2D}

\author[mymainaddress]{Shicheng Liu}
\ead{lsc22@mails.jlu.edu.cn}

\author[mysecondaddress]{Xiangyun Meng}
\ead{xymeng1@bjtu.edu.cn}

\author[mymainaddress]{Qilong Zhai\corref{mycorrespondingauthor}}
\cortext[mycorrespondingauthor]{Corresponding author}
\ead{zhaiql@jlu.edu.cn}

\address[mymainaddress]{School of Mathematics, Jilin University, Changchun {130012}, China}
\address[mysecondaddress]{School of Mathematics and Statistics, Beijing Jiaotong University, Beijing {100044}, China} 

\begin{abstract}
	In this paper, we apply the weak Galerkin (WG) finite element method to solve the singularly perturbed fourth-order boundary value problem in 2D domain. A Shishkin mesh is used to ensure that the method exhibits uniform convergence, regardless of the singular perturbation parameter. Asymptotically optimal order error estimate in a $H^2$ discrete norm is established for the corresponding WG solutions. Numerical tests are presented to verify the theory of convergence.
\end{abstract}

%


\end{frontmatter}


\section{Introduction}
\label{section:introduction}
This paper is concerned with a singularly perturbed fourth-order boundary value problem in a square region $\Omega$. The problem is solved by the weak Galerkin (WG) method with a Shishkin mesh, as follows:
\begin{align}
  \varepsilon ^{2}\Delta ^{2}u-\Delta u&=f, \quad \text{in}~\Omega,\label{1.1}\\
  u&=0, \quad \text{on}~\partial\Omega,\label{1.2}\\
  \frac{\partial u}{\partial \mathbf{n}}&=0, \quad \text{on}~\partial\Omega,\label{1.3}
\end{align}
with a positive parameter $\varepsilon$ satisfying $0< \varepsilon \ll 1$ and $f \in L^2(\Omega)$. This problem is used for thin elastic plates clamped in tension. $f$ represents the transverse load, $\varepsilon$ symbolizes the ratio of bending stiffness to tensile stiffness of the plate, and the function's solution, denoted as $u$, represents the displacement of the plate. This problem emerges in the investigation of the linearization of the fourth-order perturbation associated with the fully nonlinear Monge-Ampère equation\cite{MR2813346,Monge_mpere}.

A variational formulation for the fourth-order equation (\ref{1.1}) with the boundary conditions (\ref{1.2}) and (\ref{1.3}) seeks $u\in H^2(\Omega)$ such that
\begin{equation}\label{1.4}
  \varepsilon^{2}(\Delta u,\Delta v)+(\nabla u,\nabla v)=(f,v),\quad\forall v\in H^{2}_{0}(\Omega),
\end{equation}
where $H^2_0(\Omega)$ is a subspace of the Sobolev space $H^2(\Omega)$ consisting of functions with vanishing value and normal derivative on $\partial \Omega$.

Research into numerical methods for singularly perturbed differential equations commenced in the early 1970s, with the frontier of research continuously expanding ever since. Bakhvalov made an important early contribution to the optimization of numerical methods by means of special meshes \cite{MR0255066} in 1969. In the early 1990s, G.I. Shishkin proposed piecewise-equidistant meshes \cite{MR1439750, Moscow}. These meshes, characterized by their very simple structure, are usually easy to analyze. Shishkin meshes for various problems and numerical methods have been studied since and they are still popular. Numerous scholars have explored the numerical analysis of analogous problems using various finite element methods, such as the mixed finite element method in \cite{mixed_FEM}, the hp finite element method in \cite{hpFEM1}, the continuous interior penalty finite element method in \cite{C_interior_penaltyFEM}, the upwind finite volume element method in \cite{FVEM_meng}, the conforming finite element method in \cite{CFEM}. And also some papers consider finite element methods on quasi-uniform meshes: a continuous interior penalty finite element method in \cite{Monge_mpere} and nonconforming finite element method in \cite{NCFEM2,NCFEM3,NCFEM1,NCFEM_meng,NCFEM4,NCFEM5,NCFEM7,NCFEM6}.

In this paper, we employ the WG method with a Shishkin mesh to investigate the convergence behavior of a fourth-order boundary value problem that exhibits singular perturbation. The WG method proves to be an effective numerical technique for the partial differential equations(PDEs). The initial proposal for its application in solving second-order elliptic problems was made by Junping Wang and Xiu Ye in \cite{wysec2}. The core concept involves establishing distinct basis functions for the interior and boundary of each partitioned element, and substituting the traditional differential operator with a discretized weak differential operator. The WG method has been applied to Stokes equations \cite{wy1302,WangZhaiZhangS2016,wangwangliu2022}, elasticity equations \cite{chenxie2016,Liu,lockingw,Yisongyang}, Maxwell's equations \cite{MLW14}, biharmonic equations \cite{MuWangYeZhang14,ZhangZhai15}, Navier-Stokes equations \cite{HMY2018,LLC18,zzlw2018}, Brinkman equations \cite{MuWangYe14,WangZhaiZhang2016,ZhaiZhangMu16},  the multigrid approach \cite{CWWY15}, the incompressible flow \cite{ZL18}, the maximum principle \cite{maximumwang2,WangYeZhaiZhang18}, the post-processing technique \cite{WangZhangZhangZhang18} and so on. For singular perturbed value problems, the WG method has also yielded some results, such as the singularly perturbed convection-diffusion problems for WG in 1D \cite{WG_1D_Bakhvalov,WG_1D_Shishkin} and 2D \cite{WG_2D_Shishkin}, the singularly perturbed biharmonic equation for WG in uniform mesh \cite{singularly_perturbed_biharmonic}.


This paper is organized as follows.
In Section 2, we introduce the Shishkin mesh and the assumptions associated. In Section 3, we give the definitions of the weak Laplacian operator and weak gradient operator. We also present WG finite element schemes for the singularly perturbed value problem. In Section 4, we introduce some local $L^2$ projection operators and give some approximation properties. In Section 5, we establish error estimates for the WG scheme in a $H^2$-equivalent discrete norm. And in Section 6, we report the results of two numerical experiments.

\section{Preliminaries and notations}


To solve problem (\ref{1.1})-(\ref{1.3}), we suppose the following assumption holds in \cite{NCFEM_meng}, which involves structuring the solution and decomposing $u$ into smooth and layered components.

\begin{assumption}\label{assume2.1}
  The solution $u$ to the singularly perturbed fourth-order boundary value problem (\ref{1.1})-(\ref{1.3}) can be expressed as the sum of its smooth and layered components, as follows:
  \begin{align*}
    u=S+\sum_{l=1}^{4}E_{l}+E_{12}+E_{23}+E_{34}+E_{41}.
  \end{align*}
  In this decomposition, $S$ represents a smooth function, each $E_{l}$ corresponds to a boundary layer component along the sides of $\bar{\Omega}$ in anti-clockwise order, and the remaining components are corner layer parts along the corners of $\bar{\Omega}$ in anti-clockwise order. Furthermore, there exists a constant $C$ for all points $(x, y) \in \bar{\Omega}$, which is independent of $x$ and $y$. This constant satisfies the following conditions for $0 \leq i+j \leq k+1$,
    \begin{align*}
        &\left\vert\frac{\partial^{i+j}S}{\partial x^{i}\partial y^{j}}\right\vert\leq C,\\
        &\left\vert\frac{\partial^{i+j}E_{1}}{\partial x^{i}\partial y^{j}}\right\vert\leq C\varepsilon^{1-j}e^{-\frac{y}{\varepsilon}},\\
        &\left\vert\frac{\partial^{i+j}E_{4}}{\partial x^{i}\partial y^{j}}\right\vert\leq C\varepsilon^{1-i}e^{-\frac{x}{\varepsilon}},\\
        &\left\vert\frac{\partial^{i+j}E_{41}}{\partial x^{i}\partial y^{j}}\right\vert\leq C\varepsilon^{1-i-j}e^{-\frac{x}{\varepsilon}}e^{-\frac{y}{\varepsilon}}.
    \end{align*}
    Moreover, the other components of the decomposition are bounded in a similar manner.
  \end{assumption}

  In order to solve the layer structure in the solution of problems (\ref{1.1})-(\ref{1.3}), a well-suited layer-adapted Shishkin mesh be considered. This mesh is refined in the layers. For a comprehensive discussion on the construction of Shishkin meshes, please refer to \cite{Robust}.

  Consider a positive integer $N \geq 4$ that is divisible by $4$. We introduce a mesh transition parameter $\lambda$ to determine the location at which the mesh switches from coarse to fine. This parameter is defined by
    \begin{equation}
        \lambda=min\left\{\alpha\varepsilon lnN,\frac{1}{4}\right\},
    \end{equation}
  where $\alpha$ is a positive constant, selected to be $k+1$ for the subsequent analysis. 

Create a piecewise equidistant mesh for the interval $[0, 1]$ by dividing it as follows: divide $[0, \lambda]$ into $N/4$ subintervals, $[\lambda, 1-\lambda]$ into $N/2$ subintervals, and $[1-\lambda, 1]$ into $N/4$ subintervals. The Shishkin mesh for the problem (\ref{1.1})-(\ref{1.3}) is formed by taking the tensor product of two such one-dimensional meshes, as illustrated in Figure \ref{figure_1}. The fine meshwidth denoted as $h$ and the coarse meshwidth denoted as $H$ in the Shishkin mesh represented by $\mathcal{T}_N$ exhibit the following characteristics:
  \begin{equation}\label{2.2}
      h=\frac{4\lambda}{N}\leq C\varepsilon N^{-1}lnN \quad\text{and}\quad H=2\frac{1-2\lambda}{N}\leq CN^{-1}
  \end{equation}
for some constant $C$. The domain $\Omega$ is divided into some subdomains, see Figure \ref{figure_1}. Denote $\Omega_{l1} \cup \Omega_{s} \cup \Omega_{l3}$ by $\Omega_{r}^{1}$ and $\Omega_{l2} \cup \Omega_{s} \cup \Omega_{l4}$ by $\Omega_{r}^{2}$.
\begin{figure}[h]
  \centering
  \includegraphics[width=12cm]{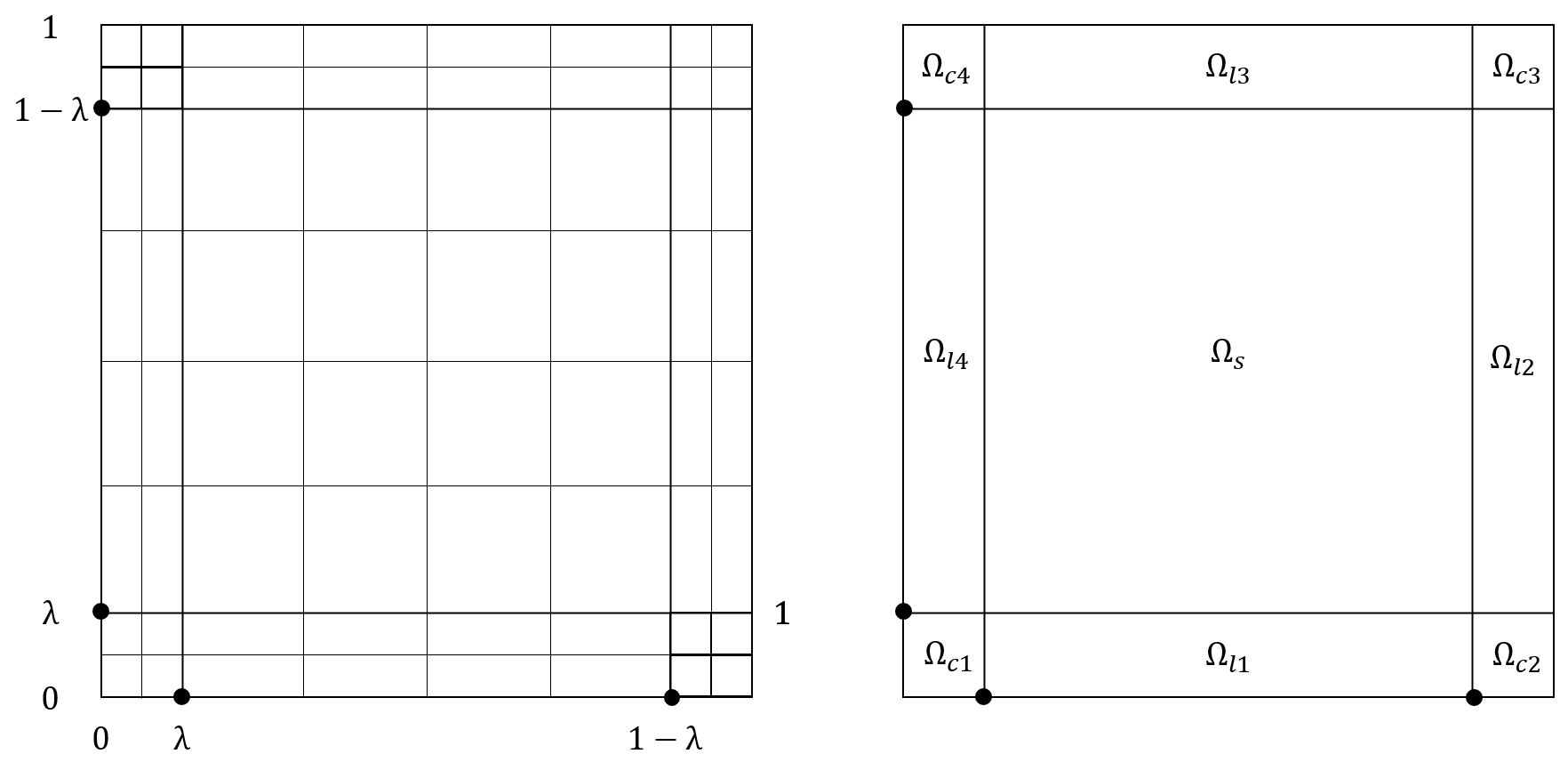}
      \caption{A rectangular Shishkin mesh with $N = 8$ and dissection of $\Omega$.}
      \label{figure_1}
\end{figure}

\section{Weak differential operator and WG scheme}
To propose the weak Galer-kin method, we introduce some key concepts. Consider a element $T$ belonging to the partition $\mathcal{T}_{N}$ with a boundary denoted by $\partial T$. Let the set of all edges in $\mathcal{T}_N$ be represented as $\mathcal{E}_N$. We define a weak function $v = \{v_0, v_b, \mathbf{v}_g\}$ on the element $T$, where $v_0 \in L^2(T)$, $v_b \in H^\frac{1}{2}(\partial T)$, and $\mathbf{v}_g \cdot \mathbf{n} \in H^\frac{1}{2}(\partial T)$, with $\mathbf{n}$ representing the outward normal direction on $\partial T$. Furthermore, the first and second components, $v_0$ and $v_b$, correspond to values of $v$ in the interior and on the boundary of $T$. The third component $\mathbf{v}_g$ is employed to approximate the gradient of $v$ along the boundary of $T$. It's important to note that each edge $e \in \mathcal{E}_N$ has a unique value for $v_b$ and $\mathbf{v}_g$. Additionally, it's worth mentioning that $v_b$ and $\mathbf{v}_g$ may not necessarily be associated with the trace of $v_0$ and $\nabla v_0$ on $\partial T$.

For any integer $k \geq 3$, we establish a local discrete weak function space for any element $T \in \mathcal{T}_N$ denoted as $W_k(T)$:
\begin{equation*}
    W_{k}(T)=\left\{v=\{v_0, v_b, \mathbf{v}_g\}: v_0\in \mathbb{Q}_{k}(T), v_b\in \mathbb{P}_{k}(e), \mathbf{v}_g\in [\mathbb{P}_{k}(e)]^{2}, e \in \partial T\right\},
\end{equation*}
where $e$ is the edge of $\partial T$, and $\mathbb{Q}_{k}$ is the space of polynomials which are of degree not exceeding $k$ with respect to each one of the variables $x$ and $y$. By extending $W_k(T)$ to encompass all element $T \in \mathcal{T}_N$, we introduce the definition of a weak Galerkin space:
$$V_N=\left\{v=\{v_0, v_b, \mathbf{v}_g\}: v|_T\in W_k(T), \forall T\in\mathcal{T}_{N}\right\}.$$
Let $V_{N}^{0}$ represent the subspace of $V_h$ where the traces vanish:
$$V_{N}^{0}=\left\{v=\{v_0, v_b, \mathbf{v}_g\}\in V_h, v_b|_e=0, \mathbf{v}_g\cdot\mathbf{n}|_e=0, e\subset\partial T\cap \partial\Omega\right\}.$$

For any $v = \{v_0, v_b, \mathbf{v}_g\}$ and a fixed integer $k \geq 3$, we define a discrete weak Laplacian operator $\Delta_{w, k}$ as a unique polynomial $\Delta_{w, k} v \in \mathbb{Q}_{k}(T)$ on $T$ satisfying the following equation:
\begin{align}\label{3.1}
  \left(\Delta_{w,k}v,\varphi\right)_T=\left(v_0, \Delta\varphi\right)_T-\left\langle v_b, \nabla\varphi\cdot\mathbf{n}\right\rangle+\left\langle \mathbf{v}_g\cdot\mathbf{n}, \varphi\right\rangle, \quad\forall\varphi\in \mathbb{Q}_k(T),
\end{align}
where $\mathbf{n}$ is the unit outward normal vector to $\partial T$. Likewise, a discrete weak gradient $\nabla_{w, k}$ is defined on $T$ as a unique polynomial $\nabla_{w, k} v \in [\mathbb{Q}_{k}(T)]^{2}$ satisfying:
\begin{align}\label{3.2}
  \left(\nabla_{w,k}v, \mathbf{q}\right)_T=-\left(v_0, \nabla\cdot \mathbf{q}\right)_T+\left\langle v_b, \mathbf{q}\cdot\mathbf{n}\right\rangle_{\partial T}, \quad\forall \mathbf{q}\in \left[\mathbb{Q}_k(T)\right]^{2}. 
\end{align} 

For simplicity, when there is no confusion, we drop the subscript $k$ in the notations $\Delta_{w,k}$ and $\nabla_{w,k}$ for the discrete weak Laplacian and the discrete weak gradient. Additionally, we introduce the following notations:
\begin{align*}
  \left(\Delta_{w}v, \Delta_{w}w\right)_{\mathcal{T}_{N}}:&=\sum_{T\in\mathcal{T}_N}\left(\Delta_{w}v, \Delta_{w}w\right)_T,\\
  \left(\nabla_{w}v, \nabla_{w}w\right)_{\mathcal{T}_{N}}:&=\sum_{T\in\mathcal{T}_N}\left(\nabla_{w}v, \nabla_{w}w\right)_T.
\end{align*}

Let us introduce a stabilizer, which is a bilinear form for any $u_N=\{u_0, u_b, \mathbf{u}_g\}$ and $v=\{v_0, v_b, \mathbf{v}_g\}$ in the space $V_h$. It is defined as follows:
\begin{align}
  \begin{aligned}
    s\left(u_h,v\right) = \sum_{T\in\mathcal{T}_N} \Big( \varepsilon^{2} h^{-1} & \left\langle\nabla u_0-\mathbf{u}_g,\nabla v_0-\mathbf{v}_g\right\rangle_{\partial T}\\
    &+\left(\varepsilon^{2}h^{-2}H^{-1} +H^{-1}\right) \left\langle u_0-u_b,v_0-v_b\right\rangle_{\partial T} \Big),
  \end{aligned}
\end{align}
where $h$ and $H$ is defined as (\ref{2.2}).

\begin{algorithm}[h!]
\caption{WG Algorithm}
To obtain a numerical approximation for (\ref{1.1})-(\ref{1.3}), we seek $u_N=\{u_0, u_b, \mathbf{u}_g\} \in V_{N}$ that satisfies the following equation:
\begin{align}\label{3.3}
  \varepsilon^{2}(\Delta_{w}u_{N},\Delta_{w}v)_{N}+(\nabla_{w}u_{N},\nabla_{w}v)_{N}+s(u_{N},v)=(f,v),
\end{align}
for any $v=\{v_0,v_b,\mathbf{v}_g\}\in V_N^0$. 
\end{algorithm}

The following lemma provides a valuable result regarding the finite element space $V^0_N$.
\begin{lemma}
  For any $v \in V^0_N$, let $\3bar v \3bar$ be given as follows:
  \begin{align}\label{3.4}
    \3bar v \3bar^2=\varepsilon^{2}(\Delta_{w}v,\Delta_{w}v)_{\mathcal{T}_{N}}+(\nabla_{w}v,\nabla_{w}v)_{\mathcal{T}_{N}}+s(v,v),
\end{align}
Then, $\3bar \cdot \3bar$ defines a norm in $V^0_N$.
\end{lemma}
\begin{proof}
  We shall only confirm the positivity property for $\3bar \cdot \3bar$. Consider $v = \{v_0, v_b, \mathbf{v}_g\} \in V^0_N$ with the assumption that $\3bar v \3bar = 0$. This implies $\Delta_{w}v = 0$ and $\nabla_{w}v = 0$ in each element $T$, while also satisfying $v_0 = v_b$ and $\nabla v_0 = \mathbf{v}_g$ on $\partial T$. Next, we will proof that $\Delta v_0 = 0$ in each element $T$. To this end, for any $\varphi \in \mathbb{Q}_k(T)$, employing the definition (\ref{3.1}) and the fact that $\Delta_{w}v = 0$, we obtain
  \begin{equation}\label{3.5}
    \begin{split}
      0 &= \left(\Delta_{w}v,\varphi\right)_T\\
        &= \left(v_0,\Delta\varphi\right)_T-\left\langle v_b,\nabla\varphi\cdot\mathbf{n}\right\rangle_{\partial T}+\left\langle\mathbf{v}_g\cdot\mathbf{n},\varphi\right\rangle_{\partial T}\\
        &= \left(\Delta v_0,\varphi\right)_T+\left\langle v_0-v_b,\nabla\varphi\cdot\mathbf{n}\right\rangle_{\partial T}+\left\langle\mathbf{v}_g\cdot\mathbf{n}-\nabla v_0\cdot\mathbf{n},\varphi\right\rangle_{\partial T}\\
        &= \left(\Delta v_0,\varphi\right)_T,
  \end{split}
  \end{equation}
  where we have applied the fact that $v_0 - v_b = 0$ and $\nabla v_0 - \mathbf{v}_g = 0$ in the final equality.
  The identity (\ref{3.5}) indicates that $\Delta v_0 = 0$ in each element $T$. Together with the conditions $v_0 = v_b$ and $\nabla v_0 = \mathbf{v}_g$ on $\partial T$, we conclude that $v$ is a globally smooth harmonic function on $\Omega$. Considering the boundary conditions $v_b = 0$ and $\mathbf{v}_g \cdot \mathbf{n} = 0$, we infer that the unique solution is $v \equiv 0$ on $\Omega$. This concludes the proof.
\end{proof}

\section{Local $L^2$ projection operators and approximation properties}
In this section, we introduce some projection operators for each element $T \in \mathcal{T}_N$. Consider the $L^2$ projection operator $\mathcal{Q}_0$, which projects onto $\mathbb{Q}_k(T)$. Additionally, for each edge $e \in \partial T$, we consider the $L^2$ projection operators $\mathcal{Q}_b$ and $\mathbf{Q}_g$ onto local polynomial spaces $\mathbb{P}_k(e)$ and $[\mathbb{P}_k(e)]^2$, respectively. We define a projection $\mathcal{Q}_N$ of $u$ into the finite element space $V_N$ such that on each
element $T$
\begin{align*}
  \mathcal{Q}_N u=\left\{\mathcal{Q}_0 u, \mathcal{Q}_b u, \mathbf{Q}_g(\nabla u)\right\}.
\end{align*}
Furthermore, let $\mathbf{Q}_{N}$ represent the local $L^2$ projection onto $[\mathbb{Q}_{k}(T)]^2$. The following lemma demonstrates that the weak Laplacian $\Delta_w$ is the polynomial projection of the classical Laplacian $\Delta$. 
\begin{lemma} \label{lemma4.1}
  On each element $T \in \mathcal{T}_{N}$, for any $v \in H^2(T)$,
  \begin{align}\label{4.1}
    \Delta_{w}(\mathcal{Q}_{N} v)=\mathcal{Q}_0(\Delta v).
  \end{align}
\end{lemma}
\begin{proof}
  For any $\tau \in \mathbb{Q}_{k}(T)$, we obtain that
  \begin{align*}
    \left(\Delta_{w}\mathcal{Q}_{N} u,\tau\right)_{T}&=\left(\mathcal{Q}_{0} u,\Delta \tau \right)_{T}-\left\langle\mathcal{Q}_{b}u,\nabla\tau\cdot\mathbf{n}\right\rangle_{\partial T}+\left\langle\mathbf{Q}_g(\nabla u)\cdot\mathbf{n},\tau\right\rangle_{\partial T}\\
    &=\left(u,\Delta\tau\right)_{T}-\left\langle u,\nabla\tau\cdot\mathbf{n}\right\rangle_{\partial T}+\left\langle\nabla u \cdot\mathbf{n},\tau\right\rangle_{\partial T}\\
    &=\left(\Delta u,\tau\right)_{T}\\
    &=\left(\mathcal{Q}_{0}\Delta u,\tau\right)_{T}.
  \end{align*}
  This concludes the proof.
\end{proof}

A similar lemma holds for the weak gradient $\nabla_{w}$, as indicated in the subsequent lemma.
\begin{lemma} \label{lemma4.2}
  On each element $T \in \mathcal{T}_{N}$, for any $v \in H^2(T)$,
  \begin{align}\label{4.2}
    \nabla_{w}(\mathcal{Q}_{N} u)=\mathbf{Q}_{N}(\nabla u).
  \end{align}
\end{lemma}
\begin{proof}
  For any $\mathbf{q} \in [\mathbb{Q}_{k}(T)]^2$, from the weak gradient definition and integration by parts, we get
  \begin{align*}
    \left(\nabla_{w}\mathcal{Q}_{N} u,\mathbf{q}\right)_{T}&=-\left(\mathcal{Q}_{0}u,\nabla\cdot \mathbf{q}\right)_{T}+\left\langle\mathcal{Q}_{b} u,\mathbf{q}\cdot\mathbf{n}\right\rangle_{\partial T}\\
    &=-\left(u,\nabla\cdot \mathbf{q}\right)_{T}+\left\langle u, \mathbf{q}\cdot\mathbf{n}\right\rangle_{\partial T}\\
    &=\left(\nabla u, \mathbf{q}\right)_{T}\\
    &=\left(\mathbf{Q}_{N}\nabla u, \mathbf{q}\right)_{T}.
  \end{align*}
  This corresponds to identity (\ref{4.2}).
\end{proof}

Now, we introduce notation that will be employed in the following lemmas. Consider any element $T$ in the partition $\mathcal{T}_{N}$. We define $\partial T_{1}$ and $ \partial T_{2}$ as the sets of element edges that are parallel to the $x$ and $y$ axes, respectively. The following lemmas are employed in the convergence analysis, and readers are directed to \cite[chapter 3.1]{oxford} for a detailed proof process.


\begin{lemma}\label{lemma4.3}
  Consider $\phi \in H^{k+1}(T)$ with $k \geq 3$. Let $\mathcal{Q}_{0} \phi$ denote the $L^2$-projection of $\phi$ onto $\mathbb{Q}_{k}(T)$. Then the following inequality estimate holds,
  \begin{align}
    &\Vert \phi-\mathcal{Q}_{0} \phi\Vert_{\partial T_{i}}\leq C\left(h_{j}^{k+\frac{1}{2}} \Vert \partial_{j}^{k+1}\phi\Vert_{T}+h_{i}^{k+1}h_{j}^{-\frac{1}{2}} \Vert\partial_{i}^{k+1}\phi\Vert_{T}+h_{j}^{\frac{1}{2}}h_{i}^{k} \Vert\partial_{i}^{k}\partial_{j}\phi\Vert_{T}\right), \label{I1} \\
    &\Vert\partial_{i}(\phi-\mathcal{Q}_{0} \phi)\Vert_{\partial T_{i}}\leq C\left(h_{i}^{k}h_{j}^{-\frac{1}{2}} \Vert\partial_{i}^{k+1}\phi\Vert_{T} +h_{i}^{k-1}h_{j}^{\frac{1}{2}}\Vert\partial_{i}^{k}\partial_{j}\phi\Vert_{T} +h_{j}^{k-\frac{1}{2}}\Vert \partial_{j}^{k}\partial_{i}\phi\Vert_{T}\right), \label{I2} \\
    &\Vert\partial_{j}(\phi-\mathcal{Q}_{0} \phi)\Vert_{\partial T_{i}}\leq C\left(h_{j}^{k-\frac{1}{2}}\Vert \partial_{j}^{k+1}\phi\Vert_{T} +h_{i}^{k}h_{j}^{-\frac{1}{2}}\Vert\partial_{i}^{k}\partial_{j}\phi\Vert_{T}\right), \label{I3} 
  \end{align}
  where $i,j\in\{1,2\}$, and $i\neq j$.
\end{lemma}
\begin{lemma}\label{lemma4.4}
  Let $v \in \mathbb{Q}_{k}(T)$ with $k \geq 3$ such that the following inequalities holds,
  \begin{align}
    \Vert v \Vert_{\partial T_{i}} \leq &C h_{j}^{-\frac{1}{2}} \Vert v \Vert_{T}, \label{I4} \\
    \Vert \partial_{i} v \Vert_{T} \leq &C h_{i}^{-1} \Vert v \Vert_{T}, \label{I5}
  \end{align}
  where $C$ is a constant only depends on $k$ and $i\in\{1,2\}$, $i\neq j$. 
\end{lemma}

By applying Lemma \ref{lemma4.3} and Lemma \ref{lemma4.4}, we can deduce the following estimates which are valuable for the convergence analysis of the WG finite element schemes (\ref{3.3}).
\begin{lemma}
  Let $k \geq 3$, $u \in  H^{k+1}(\Omega)$. There exists a constant $C$ such that the following estimates hold true,
  \begin{align}
    &\sum_{i=1,2} \left( \sum_{T\in\mathcal{T}_{N}}\Vert\Delta u-\mathcal{Q}_{0}\Delta u\Vert_{\partial T_{i}}^{2} \right)^{\frac{1}{2}} \leq C \left(N^{-(k-\frac{3}{2})} +\varepsilon^{-1} N^{-(k-\frac{3}{2})} \ln^{k-\frac{3}{2}}N +\varepsilon^{-1} N^{\frac{1}{2}-\alpha}\right), \label{l1} \\
    &\sum_{i=1,2} \left( \sum_{T\in\mathcal{T}_{N}} \Vert\nabla\left(\Delta u-\mathcal{Q}_{0} \Delta u\right)\Vert_{\partial T_{i}}^{2} \right)^{\frac{1}{2}} \leq C \left(N^{-(k-\frac{5}{2})} +\varepsilon^{-2} N^{-(k-\frac{5}{2})} \ln^{k-\frac{5}{2}}N +\varepsilon^{-2} N^{\frac{3}{2}-\alpha}\right), \label{l2} \\
    &\sum_{i=1,2} \left( \sum_{T\in\mathcal{T}_{N}} \Vert\nabla u-\mathbf{Q}_{N}\nabla u\Vert_{\partial T_{i}}^{2} \right)^{\frac{1}{2}} \leq C \left(N^{-(k-\frac{1}{2})} +N^{-(k-\frac{1}{2})} \ln^{k-\frac{1}{2}}N +N^{\frac{1}{2}-\alpha}\right), \label{l3} \\
    &\sum_{i=1,2} \left( \sum_{T\in\mathcal{T}_{N}} \Vert\nabla (\mathcal{Q}_{0}u- u) \Vert_{\partial T_{i}}^{2} \right)^{\frac{1}{2}} \leq C \left(N^{-(k-\frac{1}{2})} +N^{-(k-\frac{1}{2})} \ln^{k-\frac{1}{2}}N +N^{\frac{3}{2}-\alpha}\right), \label{l4} \\
    &\sum_{i=1,2} \left( \sum_{T\in\mathcal{T}_{N}} \Vert \mathcal{Q}_{0}u-u\Vert_{\partial T_{i}}^{2} \right)^{\frac{1}{2}} \leq C \left(N^{-(k+\frac{1}{2})} +\varepsilon N^{-(k+\frac{1}{2})} \ln^{k+\frac{1}{2}}N +\varepsilon N^{\frac{1}{2}-\alpha}\right). \label{l5}
  \end{align}
\end{lemma}


\begin{proof}
  To derive (\ref{l1}), we estimate $\sum_{T\in\mathcal{T}{N}}\Vert\Delta u-\mathcal{Q}_{0}\Delta u\Vert_{\partial T_{i}}$ by breaking down the function $u$ in Assumption \ref{assume2.1}. Each term in the decomposition will be considered individually. To begin, we can apply inequality (\ref{I1}) to obtain
  \begin{align*}
    \left( \sum_{T\in\Omega_{c}} \Vert\Delta S-\mathcal{Q}_{0}\Delta S\Vert_{\partial T_{i}}^{2} \right)^{\frac{1}{2}} \leq &C h^{k-\frac{3}{2}}\cdot N \cdot h \leq C N^{-(k-\frac{3}{2})},\\
    \left( \sum_{T\in\Omega_{s}} \Vert\Delta S-\mathcal{Q}_{0}\Delta S\Vert_{\partial T_{i}}^{2} \right)^{\frac{1}{2}} \leq &C H^{k-\frac{3}{2}}\cdot N \cdot H \leq C N^{-(k-\frac{3}{2})},\\
    \left( \sum_{T\in\Omega_{l}} \Vert\Delta S-\mathcal{Q}_{0}\Delta S\Vert_{\partial T_{i}}^{2} \right)^{\frac{1}{2}} \leq &C N\cdot h_{1}^{\frac{1}{2}}h_{2}^{\frac{1}{2}} \left(h_{j}^{k-\frac{3}{2}} +h_{j}^{-\frac{1}{2}}h_{i}^{k-1} +h_{j}^{\frac{1}{2}}h_{i}^{k-2}
    \right)\\
    \leq &C N^{-(k-\frac{3}{2})},
  \end{align*}
  where we have used the fact that $h_{i} = h$ in the domain $\Omega_{c}$ and $h_{i} = H$ in the domain $\Omega_{s}$, for $i = 1, 2$. Next, we provide estimates only for the sets of element edges parallel to the $x$ axis, as the orther part follows a similar way. Considering the boundary layer $E_{1}$, we obtain
  \begin{align*}
    \left( \sum_{T\in\Omega_{c}} \Vert\Delta E_{1}-\mathcal{Q}_{0}\Delta E_{1}\Vert_{\partial T_{1}}^{2} \right)^{\frac{1}{2}} \leq &C h^{k-\frac{3}{2}} \cdot \varepsilon^{-k} \left(\int_{\Omega_{c}} e^{-\frac{2y}{\varepsilon}} \,dxdy\right)^{\frac{1}{2}}\\
    \leq &C \varepsilon^{-k} \cdot h^{k-\frac{3}{2}} \cdot \varepsilon \ln^{\frac{1}{2}}N\\
    \leq &C \varepsilon^{-\frac{1}{2}} N^{-(k-\frac{3}{2})} \ln^{k-1}N,\\
    \left( \sum_{T\in\Omega_{l}} \Vert\Delta E_{1}-\mathcal{Q}_{0}\Delta E_{1}\Vert_{\partial T_{1}}^{2} \right)^{\frac{1}{2}} \leq &C \left(\right. \varepsilon^{2-k} h_{2}^{k-\frac{3}{2}} +\varepsilon h_{1}^{k-1}h_{2}^{-\frac{1}{2}} +h_{1}^{k-2}h_{2}^{\frac{1}{2}} +\varepsilon^{-k} h_{2}^{k-\frac{3}{2}}\\
    &  +\varepsilon^{-1} h_{1}^{k-1}h_{2}^{-\frac{1}{2}} +\varepsilon^{-2} h_{1}^{k-2}h_{2}^{\frac{1}{2}} \left.\right)\left(\int_{\Omega_{l}} e^{-\frac{2y}{\varepsilon}} \,dxdy\right)^{\frac{1}{2}}
    \\
    \leq &C \varepsilon^{-\frac{3}{2}} N^{-(k-\frac{3}{2})} \ln^{k-\frac{3}{2}}N \cdot \varepsilon^{\frac{1}{2}}\\
    \leq &C \varepsilon^{-1} N^{-(k-\frac{3}{2})} \ln^{k-\frac{3}{2}}N,
  \end{align*}
  where we have used the inequality (\ref{I1}). As for the region $\Omega_{r}^{2}$ that remains in the partition, by applying inequality (\ref{I4}), we derive
  \begin{align*}
    \left( \sum_{T\in\Omega_{r}^{2}} \Vert\Delta E_{1}-\mathcal{Q}_{0}\Delta E_{1}\Vert_{\partial T_{1}}^{2} \right)^{\frac{1}{2}}
    \leq &C \sum_{T\in\Omega_{r}^{2}} \Vert\Delta E_{1} \Vert_{\partial T_{1}} +\sum_{T\in\Omega_{r}^{2}} h_{2}^{-\frac{1}{2}} \Vert \mathcal{Q}_{0}\Delta E_{1}\Vert_{T}\\
    \leq &C (\varepsilon +\varepsilon^{-1} )\left(N \int_{0}^{1} e^{-\frac{2y}{\varepsilon}}\,dx \right)^{\frac{1}{2}} +H^{-\frac{1}{2}}\left\Vert \Delta E_{1} \right\Vert_{\Omega_{r}^{2}}\\
    \leq &C \left(\varepsilon^{-1} N^{\frac{1}{2}-\alpha} +H^{-\frac{1}{2}}(\varepsilon^{\frac{3}{2}} N^{-\alpha} +\varepsilon^{-\frac{1}{2}} N^{-\alpha})\right)\\
    \leq &C \varepsilon^{-1} N^{\frac{1}{2}-\alpha}.
  \end{align*}
  A similar bound can be readily obtained for $E_2, E_3$, and $E_4$. Let's focus on estimating $E_{41}$ for the concer layers, as the other concer layers in the decomposition from Assumption \ref{assume2.1} follow a similar way. By applying inequalities (\ref{I1}) and (\ref{I4}), we arrive at
  \begin{align*}
    \left( \sum_{T\in\Omega_{c}} \Vert\Delta E_{41}-\mathcal{Q}_{0}\Delta E_{41}\Vert_{\partial T_{i}}^{2} \right)^{\frac{1}{2}} \leq &C\varepsilon^{-k} h^{k-\frac{3}{2}} \cdot \left(\int_{\Omega_{c}} e^{-\frac{2x}{\varepsilon}} e^{-\frac{2y}{\varepsilon}} \,dxdy\right)^{\frac{1}{2}}\\
    \leq &C \varepsilon^{-\frac{1}{2}} N^{-(k-\frac{3}{2})} \ln^{k-\frac{3}{2}}N.\\
    \left( \sum_{T\in \Omega_{r}^{1}} \Vert\Delta E_{41}-\mathcal{Q}_{0}\Delta E_{41}\Vert_{\partial T_{1}}^{2} \right)^{\frac{1}{2}}
    \leq &C \sum_{T\in \Omega_{r}^{1}} \left(\left\Vert \Delta E_{41}\right\Vert_{\partial T_{1}} +h_{2}^{-\frac{1}{2}} \left\Vert \mathcal{Q}_{0}\Delta E_{41}\right\Vert_{T} \right)\\
    \leq &C \varepsilon^{-1} \left(N \int_{\lambda}^{1-\lambda} e^{-\frac{2x}{\varepsilon}} e^{-\frac{2y}{\varepsilon}}\,dx\right)^{\frac{1}{2}} +h^{-\frac{1}{2}} \left\Vert \Delta E_{41} \right\Vert_{\Omega_{r}^{1}}\\
    \leq &C \left(\varepsilon^{-1} \cdot \varepsilon^{\frac{1}{2}} N^{\frac{1}{2}-\alpha} +h^{-\frac{1}{2}} \cdot N^{-\alpha}\right)\\
    \leq &C \varepsilon^{-\frac{1}{2}} N^{\frac{1}{2}-\alpha},\\
    \left( \sum_{T\in \Omega_{r}^{2}} \Vert\Delta E_{41}-\mathcal{Q}_{0}\Delta E_{41}\Vert_{\partial T_{1}}^{2} \right)^{\frac{1}{2}}
    \leq &C \sum_{T\in \Omega_{r}^{2}} \left(\left\Vert \Delta E_{41}\right\Vert_{\partial T_{1}} +h_{2}^{-\frac{1}{2}} \left\Vert \mathcal{Q}_{0}\Delta E_{41}\right\Vert_{T} \right)\\
    \leq &C \varepsilon^{-1} \left(N \int_{0}^{1} e^{-\frac{2x}{\varepsilon}} e^{-\frac{2y}{\varepsilon}}\,dx\right)^{\frac{1}{2}} +H^{-\frac{1}{2}} \left\Vert \Delta E_{41} \right\Vert_{\Omega_{r}^{2}}\\
    \leq &C \left(\varepsilon^{-1} \cdot \varepsilon^{\frac{1}{2}} N^{\frac{1}{2}-\alpha} +H^{-\frac{1}{2}} \cdot N^{-\alpha}\right)\\
    \leq &C \varepsilon^{-\frac{1}{2}} N^{\frac{1}{2}-\alpha}.
  \end{align*}
  By combining the aforementioned proofs, we establish inequality (\ref{l1}). Likewise, for inequalities (\ref{l2}), (\ref{l3}), (\ref{l4}), and (\ref{l5}), the proof follows a similar way as above. Therefore, we omit the detailed explanation.
\end{proof}  

\section{Error estimate}
In this section, the objective is to provide error estimates for the WG solution $u_N$ obtained from (\ref{3.3}).

\subsection{Error equation}
We introduce notation used in error analysis to represent the error between the finite element solution and the $L^2$ projection of the exact solution, as follows
\begin{align*}
  e_N=\mathcal{Q}_{N}u-u_{N}=\{e_{0}, e_{b}, \mathbf{e}_{g}\}. 
\end{align*}
The convergence analysis relies on the error equation, and in the following lemma, we will establish an equation that the error $e_N$ satisfies.

\begin{lemma}
  Suppose $u$ and $u_N = \{u_0, u_b, \mathbf{u}_{g}\} \in V_N$ represent the solutions of (\ref{1.1}) and (\ref{3.3}), respectively. Then for any $v \in V_{N}^{0}$, we have
  \begin{align}\label{5.1}
    a(e_{N},v)=l_{1}(u,v)-l_{2}(u,v)+l_{3}(u,v)+s(\mathcal{Q}_{N}u,v),
\end{align}
where
\begin{align*}
    &l_{1}(u,v)=\sum_{T\in\mathcal{T}_N}\varepsilon^2\left\langle\Delta u-\mathcal{Q}_{0}\Delta u,\left(\nabla v_0-\mathbf{v}_g\right)\cdot\mathbf{n}\right\rangle_{\partial T},\\
    &l_{2}(u,v)=\sum_{T\in\mathcal{T}_N}\varepsilon^2\left\langle\nabla\left(\Delta u-\mathcal{Q}_{0}\Delta u\right)\cdot\mathbf{n},v_0-v_b\right\rangle_{\partial T},\\
    &l_{3}(u,v)=\sum_{T\in\mathcal{T}_{N}}\left\langle\left(\nabla u-\mathbf{Q}_{N}\nabla u\right)\cdot\mathbf{n},v_0-v_b\right\rangle_{\partial T}.
\end{align*}
\end{lemma}
\begin{proof}
  Using the definition of weak Laplacian (\ref{3.2}), integration by parts and Lemma \ref{lemma4.1}, for any $v \in V_{N}^{0}$, we yield
    \begin{align*}
        \left(\Delta_{w}\mathcal{Q}_{N}u,\Delta_{w}v\right)_{T}
        &=\left(v_0,\Delta\left(\Delta_{w}\mathcal{Q}_{N}u\right)\right)_{T}+\left\langle\mathbf{v}_g\cdot\mathbf{n},\Delta_{w}\mathcal{Q}_{N}u\right\rangle_{\partial T}-\left\langle v_b,\nabla\left(\Delta_{w}\mathcal{Q}_{N}u\right)\cdot\mathbf{n}\right\rangle_{\partial T}\\
        &=\left(\Delta v_0,\Delta_{w}\mathcal{Q}_{N}u\right)_{T}+\left\langle v_0,\nabla\left(\Delta_{w}\mathcal{Q}_{N}u\right)\cdot\mathbf{n}\right\rangle_{\partial T}-\left\langle\nabla v_0\cdot\mathbf{n},\Delta_{w}\mathcal{Q}_{N}u\right\rangle_{\partial T}\\
        &\quad+\left\langle\mathbf{v}_g\cdot\mathbf{n},\Delta_{w}\mathcal{Q}_{N}u\right\rangle_{\partial T}-\left\langle v_b,\nabla\left(\Delta_{w}\mathcal{Q}_{N}u\right)\cdot\mathbf{n}\right\rangle_{\partial T}\\
        &=\left(\Delta v_0,\Delta_{w}\mathcal{Q}_{N}u\right)_{T}+\left\langle v_{0}-v_{b},\nabla\left(\Delta_{w}\mathcal{Q}_{N}u\right)\cdot\mathbf{n}\right\rangle_{\partial T}-\left\langle\left(\nabla v_0-\mathbf{v}_g\right)\cdot\mathbf{n},\Delta_{w}\mathcal{Q}_{N}u\right\rangle_{\partial T}\\
        &=\left(\Delta v_0,\mathcal{Q}_{0}\Delta u\right)_{T}+\left\langle v_{0}-v_{b},\nabla\left(\mathcal{Q}_{0}\Delta u\right)\cdot\mathbf{n}\right\rangle_{\partial T}-\left\langle\left(\nabla v_0-\mathbf{v}_g\right)\cdot\mathbf{n},\mathcal{Q}_{0}\Delta u\right\rangle_{\partial T}\\
        &=\left(\Delta u,\Delta v_0\right)_{T}+\left\langle v_{0}-v_{b},\nabla\left(\mathcal{Q}_{0}\Delta u\right)\cdot\mathbf{n}\right\rangle_{\partial T}-\left\langle\left(\nabla v_0-\mathbf{v}_g\right)\cdot\mathbf{n},\mathcal{Q}_{0}\Delta u\right\rangle_{\partial T},
    \end{align*}
    which implies that
    \begin{equation}\label{5.2}
    \begin{split}
        \sum_{T\in\mathcal{T}_h}\left(\Delta u,\Delta v_0\right)_{T}=&\left(\Delta_{w}\mathcal{Q}_{N}u,\Delta_{w}v\right)_{N}-\sum_{T\in\mathcal{T}_h}\left\langle v_0-v_b,\nabla\left(\mathcal{Q}_{0}\Delta u\right)\cdot\mathbf{n}\right\rangle_{\partial T}\\
        &+\sum_{T\in\mathcal{T}_h}\left\langle\left(\nabla v_0-\mathbf{v}_g\right)\cdot\mathbf{n},\mathcal{Q}_{0}\Delta u\right\rangle_{\partial T}.
    \end{split}
    \end{equation}
    Likewise, we deduce from integration by parts and Lemma \ref{lemma4.2} the following
    \begin{align*}
        \left(\nabla_{w}\mathcal{Q}_{N}u, \nabla_{w}v\right)_{T}&=-\left(v_0, \nabla\cdot\nabla_w\mathcal{Q}_{N}u\right)_{T}+\left\langle v_b, \nabla_w\mathcal{Q}_{N}u\cdot\mathbf{n}\right\rangle_{\partial T}\\
        &=\left(\nabla v_{0}, \nabla_w\mathcal{Q}_{N}u\right)_{T}+\left\langle v_b-v_0, \nabla_w\mathcal{Q}_{N}u\cdot\mathbf{n}\right\rangle_{\partial T}\\
        &=\left(\nabla v_{0}, \mathbf{Q}_{N}\nabla u\right)_{T}+\left\langle v_b-v_0, \mathbf{Q}_{N}\nabla u\cdot\mathbf{n}\right\rangle_{\partial T}\\
        &=\left(\nabla v_{0}, \nabla u\right)_{T}-\left\langle v_{0}-v_{b}, \left(\mathbf{Q}_{N}\nabla u\right)\cdot\mathbf{n}\right\rangle_{\partial T},
    \end{align*}
    which implies that
    \begin{equation}\label{5.3}
    \begin{split}
        \sum_{T\in\mathcal{T}_h}\left(\nabla u,\nabla v_0\right)_{T}=\left(\nabla_{w}\mathcal{Q}_{N}u,\nabla_{w}v\right)_{N}+\sum_{T\in\mathcal{T}_h}\left\langle v_0-v_b,\mathbf{Q}_{N}\nabla u\cdot\mathbf{n}\right\rangle_{\partial T}.
    \end{split}
    \end{equation}
    Test equation (\ref{1.1}) with the vector $v_0$ of $v=\{v_0, v_b, \mathbf{v}g\} \in V_{N}^{0}$, we find
    \begin{align*}
      \varepsilon^2\left(\Delta^2u,v_0\right)-\left(\Delta u,v_0\right)=\left(f,v_0\right).
    \end{align*}
    Using the boundary conditions that $\mathbf{v}_{g}\cdot\mathbf{n}$ and $v_b$ vanish on $\partial \Omega$, along with integration by parts, we derive
    \begin{align*}
        \varepsilon^2\left(\Delta^2u,v_0\right)-\left(\Delta u,v_0\right)
        =&\sum_{T\in\mathcal{T}_h} \Big[ \varepsilon^{2}\left(\Delta u,\Delta v_0\right)_{T}-\varepsilon^{2}\left\langle\Delta u,\nabla v_{0}\cdot\mathbf{n}\right\rangle_{\partial T}+\left(\nabla u,\nabla v_0\right)_{T}\\
        &+\varepsilon^{2}\left\langle\nabla(\Delta u)\cdot\mathbf{n},v_{0}\right\rangle_{\partial T}-\left\langle\nabla u\cdot\mathbf{n},v_0\right\rangle_{\partial T} \Big]\\
        =&\sum_{T\in\mathcal{T}_h} \Big[ \varepsilon^{2}\left(\Delta u,\Delta v_0\right)_{T}-\varepsilon^{2}\left\langle\Delta u,\left(\nabla v_{0}-\mathbf{v}_{g}\right)\cdot\mathbf{n}\right\rangle_{\partial T}+\left(\nabla u,\nabla v_0\right)_{T}\\
        &+\varepsilon^{2}\left\langle\nabla(\Delta u)\cdot\mathbf{n},v_{0}-v_{b}\right\rangle_{\partial T}-\left\langle\nabla u\cdot\mathbf{n},v_0-v_b\right\rangle_{\partial T} \Big].
    \end{align*}
    Upon applying the previously mentioned equation together with (\ref{5.2}) and (\ref{5.3}), we get
    \begin{align*}
      \varepsilon^2\left(\Delta_{w}\mathcal{Q}_{N}u,\Delta_{w}v\right)_{N}+\left(\nabla_{w}\mathcal{Q}_{N}u,\nabla_{w}v\right)_{N}=\left(f,v_0\right)+l_{1}(u,v)-l_{2}(u,v)+l_{3}(u,v).
    \end{align*}
    By adding $s\left(\mathcal{Q}_{N}u,v\right)$ to both sides of the above equation, we arrive at
    \begin{equation}\label{5.4}
        a\left(\mathcal{Q}_{N}u,v\right)=\left(f,v_0\right)+l_{1}(u,v)-l_{2}(u,v)+l_{3}(u,v)+s\left(\mathcal{Q}_{N}u,v\right).
    \end{equation}
    Subtracting (\ref{3.3}) from (\ref{5.4}) yields the error equation as follows
    \begin{align*}
      a(e_{N},v)=l_{1}(u,v)-l_{2}(u,v)+l_{3}(u,v)+s(\mathcal{Q}_{N}u,v),
    \end{align*}
    for all $v\in V_{N}^{0}$. This completes the derivation of (\ref{5.1}).
\end{proof}

\subsection{Error estimate}
The following theorem is the estimate for the error function $e_{N}$ in the triple-bar norm (\ref{3.4}), which is an $H^2$-equivalent norm in $V_{N}^{0}$.
\begin{theorem}
  Consider the weak Galerkin finite element solution from (\ref{3.3}), denoted as $u_{N} \in V_{N}$. Assuming that $u \in H^{k+1}(\Omega)$, it follows that there exists a constant $C$ such that
  \begin{equation}\label{5.5}
    \3bar e_{N} \3bar \leq C \left( N^{-(k-1)} \ln^{k-\frac{3}{2}}N \right).
  \end{equation}
\end{theorem}
\begin{proof}
  Upon substituting $v = e_{N}$ into the error equation (\ref{5.1}), we derive the following equation,
  \begin{equation}\label{5.6}
    \3bar e_{N}\3bar^{2}=l_{1}(u,e_{N})-l_{2}(u,e_{N})+l_{3}(u,e_{N})+s(\mathcal{Q}_{N}u,e_{N}).
\end{equation}
Using the Cauchy-Schwarz inequality, the meshwidth characteristics (\ref{2.2}) and the inequality (\ref{l1}), we derive
\begin{equation}\label{5.7}
  \begin{split}
      l_{1}\left(u,e_{N}\right) &= \sum_{T\in\mathcal{T}_{N}}\varepsilon^{2}\left\langle\Delta u-\mathcal{Q}_{0}\Delta u,\left(\nabla e_0-\mathbf{e}_g\right)\cdot\mathbf{n}\right\rangle_{\partial T}\\
      &\leq \sum_{i=1,2} \left(\sum_{T\in\mathcal{T}_{N}} \varepsilon^{2}h \Vert\Delta u-\mathcal{Q}_{0}\Delta u\Vert_{\partial T_{i}}^{2}\right)^{\frac{1}{2}}\left(\sum_{T\in\mathcal{T}_{N}} \varepsilon^{2} h^{-1} \Vert\nabla e_0-\mathbf{e}_g\Vert_{\partial T_{i}}^{2}\right)^{\frac{1}{2}}\\
      &\leq C \varepsilon h^{\frac{1}{2}} \sum_{i=1,2} \left( \sum_{T\in\mathcal{T}_{N}}\Vert\Delta u-\mathcal{Q}_{0}\Delta u\Vert_{\partial T_{i}}^{2} \right)^{\frac{1}{2}} \3bar e_{N}\3bar\\
      &\leq C \varepsilon h^{\frac{1}{2}} \left(N^{-(k-\frac{3}{2})} +\varepsilon^{-1} N^{-(k-\frac{3}{2})} \ln^{k-\frac{3}{2}}N +\varepsilon^{-1} N^{\frac{1}{2}-\alpha}\right) \3bar e_{N}\3bar\\
      &\leq C \left(\varepsilon N^{-(k-1)} +\varepsilon^{\frac{1}{2}} N^{-(k-1)} \ln^{k-1}N +\varepsilon^{\frac{1}{2}} N^{\frac{1}{2}-\alpha}\right) \3bar e_{N}\3bar.
  \end{split}
\end{equation}
From both the Cauchy-Schwarz inequality and (\ref{2.2}) and the inequality (\ref{l2}), it can be deduced that
\begin{equation}\label{5.8}
  \begin{split}
      l_{2}\left(u,e_{N}\right)&=\sum_{T\in\mathcal{T}_{N}}\varepsilon^{2}\left\langle\nabla\left(\Delta u-\mathcal{Q}_{0} \Delta u\right)\cdot\mathbf{n},e_0-e_b\right\rangle_{\partial T}\\
      &\leq \sum_{i=1,2}\left(\sum_{T\in\mathcal{T}_{N}}\varepsilon^{2}h^{2}H \Vert\nabla\left(\Delta u-\mathcal{Q}_{0} \Delta u\right)\Vert_{\partial T_{i}}^{2}\right)^{\frac{1}{2}} \left(\sum_{T\in\mathcal{T}_{N}} \varepsilon^{2} h^{-2}H^{-1}\Vert e_0-e_b\Vert_{\partial T_{i}}^{2}\right)^{\frac{1}{2}}\\
      &\leq C \varepsilon hH^{\frac{1}{2}} \sum_{i=1,2} \left( \sum_{T\in\mathcal{T}_{N}} \Vert\nabla\left(\Delta u-\mathcal{Q}_{0} \Delta u\right)\Vert_{\partial T_{i}}^{2} \right)^{\frac{1}{2}} \3bar e_{N}\3bar\\
      &\leq C \varepsilon hH^{\frac{1}{2}} \left(N^{-(k-\frac{5}{2})} +\varepsilon^{-2} N^{-(k-\frac{5}{2})} \ln^{k-\frac{5}{2}}N +\varepsilon^{-2} N^{\frac{3}{2}-\alpha}\right) \3bar e_{N}\3bar\\
      &\leq C \left(\varepsilon N^{-(k-1)} +N^{-(k-1)} \ln^{k-\frac{3}{2}}N +N^{1-\alpha}\right) \3bar e_{N}\3bar.
  \end{split}
\end{equation}
Similarly, it follows from the Cauchy-Schwarz inequality and (\ref{2.2}) and (\ref{l3}) that
\begin{equation}\label{5.9}
  \begin{split}
      l_{3}\left(u,e_{N}\right)&=\sum_{T\in\mathcal{T}_{N}}\left\langle\left(\nabla u-\mathbf{Q}_{N}\nabla u\right)\cdot\mathbf{n},e_0-e_b\right\rangle_{\partial T}\\
      &\leq \sum_{i=1,2}\left(\sum_{T\in\mathcal{T}_{N}} H \Vert\nabla u-\mathbf{Q}_{N}\nabla u\Vert_{\partial T_{i}}^{2}\right)^{\frac{1}{2}}\left(\sum_{T\in\mathcal{T}_{N}} H^{-1} \Vert e_0-e_b\Vert_{\partial T_{i}}^{2}\right)^{\frac{1}{2}}\\
      &\leq C H^{\frac{1}{2}} \sum_{i=1,2} \left( \sum_{T\in\mathcal{T}_{N}} \Vert\nabla u-\mathbf{Q}_{N}\nabla u\Vert_{\partial T_{i}}^{2} \right)^{\frac{1}{2}} \3bar e_{N}\3bar\\
      &\leq C H^{\frac{1}{2}} \left(N^{-(k-\frac{1}{2})} +N^{-(k-\frac{1}{2})} \ln^{k-\frac{1}{2}}N +N^{\frac{1}{2}-\alpha}\right) \3bar e_{N}\3bar\\
      &\leq C \left(N^{-k} +N^{-k} \ln^{k-\frac{1}{2}}N +N^{-\alpha}\right) \3bar e_{N}\3bar.
  \end{split}
\end{equation}
In the same way, considering $s(\mathcal{Q}_{N}u,e_{N})$, it follows from the Cauchy-Schwarz inequality and (\ref{2.2}) and (\ref{l4})-(\ref{l5}) that
\begin{align}
  &\bigg\vert\sum_{i=1,2}\sum_{T\in\mathcal{T}_{N}}\varepsilon^{2} h^{-1}\left\langle\nabla \mathcal{Q}_{0}u-\mathbf{Q}_g\nabla u,\nabla e_0-\mathbf{e}_g\right\rangle_{\partial T_{i}}\bigg\vert \notag \\
  &\leq \sum_{i=1,2}\left(\sum_{T\in\mathcal{T}_{N}}\varepsilon^{2} h^{-1}\Vert\nabla \mathcal{Q}_{0}u-\nabla u\Vert_{\partial T_{i}}^{2}\right)^{\frac{1}{2}}\left(\sum_{T\in\mathcal{T}_{N}}\varepsilon^{2} h^{-1}\Vert \nabla e_0-\mathbf{e}_g\Vert_{\partial T_{i}}^{2}\right)^{\frac{1}{2}} \notag \\
  &\leq C \varepsilon h^{-\frac{1}{2}} \sum_{i=1,2} \left( \sum_{T\in\mathcal{T}_{N}} \Vert\nabla \mathcal{Q}_{0}u-\nabla u\Vert_{\partial T_{i}}^{2} \right)^{\frac{1}{2}} \3bar e_{N}\3bar \notag \\
  &\leq C \varepsilon h^{-\frac{1}{2}} \left(N^{-(k-\frac{1}{2})} +N^{-(k-\frac{1}{2})} \ln^{k-\frac{1}{2}}N +N^{\frac{3}{2}-\alpha}\right) \3bar e_{N}\3bar \notag \\
  &\leq C \left(\varepsilon^{\frac{1}{2}} N^{-(k-1)} +\varepsilon^{\frac{1}{2}} N^{-(k-1)} \ln^{k-1}N +\varepsilon^{\frac{1}{2}} N^{2-\alpha}\right) \3bar e_{N}\3bar, \label{5.10} \\
  &\bigg\vert\sum_{i=1,2}\sum_{T\in\mathcal{T}_{N}} \varepsilon^{2}h^{-2}H^{-1} \left\langle\mathcal{Q}_{0}u-\mathcal{Q}_{b}u,e_0-e_b\right\rangle_{\partial T_{i}}\bigg\vert \notag \\
  &\leq \sum_{i=1,2}\left(\sum_{T\in\mathcal{T}_{N}} \varepsilon^{2}h^{-2}H^{-1} \Vert \mathcal{Q}_{0} u -u \Vert_{\partial T_{i}}^{2}\right)^{\frac{1}{2}}\left(\sum_{T\in\mathcal{T}_{N}} \varepsilon^{2}h^{-2}H^{-1} \Vert e_0-e_b\Vert_{\partial T_{i}}^{2}\right)^{\frac{1}{2}} \notag \\
  &\leq C \varepsilon h^{-1}H^{-\frac{1}{2}} \sum_{i=1,2} \left(\sum_{T\in\mathcal{T}_{N}} \Vert \mathcal{Q}_{0} u -u \Vert_{\partial T_{i}}^{2} \right)^{\frac{1}{2}} \3bar e_{N}\3bar \notag \\
  &\leq C \varepsilon h^{-1}H^{-\frac{1}{2}} \left(N^{-(k+\frac{1}{2})} +\varepsilon N^{-(k+\frac{1}{2})} \ln^{k+\frac{1}{2}}N +\varepsilon N^{\frac{1}{2}-\alpha}\right) \notag \\
  &\leq C \left(N^{-(k-1)} +\varepsilon N^{-(k-1)} \ln^{k-\frac{1}{2}}N +\varepsilon N^{2-\alpha}\right), \label{5.11}
\end{align}
and
\begin{align}\label{5.12}
  &\bigg\vert\sum_{i=1,2}\sum_{T\in\mathcal{T}_{N}} H^{-1} \left\langle\mathcal{Q}_{0}u-\mathcal{Q}_{b}u,e_0-e_b\right\rangle_{\partial T_{i}}\bigg\vert \notag \\
  &\leq \sum_{i=1,2}\left(\sum_{T\in\mathcal{T}_{N}} H^{-1} \Vert \mathcal{Q}_{0}u-u\Vert_{\partial T_{i}}^{2}\right)^{\frac{1}{2}}\left(\sum_{T\in\mathcal{T}_{N}} H^{-1} \Vert e_0-e_b\Vert_{\partial T_{i}}^{2}\right)^{\frac{1}{2}} \notag \\
  &\leq C H^{-\frac{1}{2}} \sum_{i=1,2} \left( \sum_{T\in\mathcal{T}_{N}} \Vert \mathcal{Q}_{0} u -u \Vert_{\partial T_{i}}^{2} \right)^{\frac{1}{2}} \3bar e_{N}\3bar \notag \\
  &\leq C H^{-\frac{1}{2}} \left(N^{-(k+\frac{1}{2})} +\varepsilon N^{-(k+\frac{1}{2})} \ln^{k+\frac{1}{2}}N +\varepsilon N^{\frac{1}{2}-\alpha}\right) \notag \\
  &\leq C \left(N^{-k} +\varepsilon N^{-k} \ln^{k+\frac{1}{2}}N +\varepsilon N^{1-\alpha}\right).
\end{align}
Substituting (\ref{5.7})-(\ref{5.12}) into (\ref{5.6}) yields
\begin{align*}
  \3bar e_{N} \3bar^{2} \leq C \left( N^{-(k-1)} \ln^{k-\frac{3}{2}}N \right) \3bar e_{N} \3bar.
\end{align*}
which implies (\ref{5.5}). This completes the proof of the theorem.
\end{proof}

\section{Numerical Experiments}
In this section we compute two numerical examples where we select $\alpha=k+1$ for creating the Shishkin mesh. Consider the singularly perturbed fourth-order problem that seeks solution $u=u(x, y)$ satisfying
\begin{align*}
  \varepsilon ^{2}\Delta ^{2}u-\Delta u&=f, \quad \text{in}~\Omega,\\
  u=\partial_{\mathbf{n}} u&=0, \quad \text{on}~\partial\Omega,
\end{align*}
where $\Omega=(0, 1)^2$, $f$ and $\varepsilon$ will choose later. Tables \ref{Exam1}-\ref{Exam4} display the errors $\3bar \mathcal{Q}_N u- u_{N} \3bar$ for several different $\varepsilon$ and $N$.
\begin{example}\label{exam6.1}
  Choose $f(x, y)$ such that the exact solution is $u(x, y)=g(x)g(y)$, where
  $$g(x)=\frac{1}{2}\left[\sin(\pi x)+\frac{\pi \varepsilon}{1-e^{-1/ \varepsilon}}\left(e^{-x/ \varepsilon}+e^{(x-1)/ \varepsilon}-1-e^{-1/ \varepsilon}\right)\right].$$
  The numerical results are shown in Table \ref{Exam1}-\ref{Exam2}.
\end{example}

  In the case of $k=3$, the error and the order of convergence on Shishkin mesh and uniform mesh are listed in Table \ref{Exam1} and Table \ref{Exam1.1}, respectively. Compared with the numerical results of the WG method for the problem on uniform mesh, our results with Shishkin mesh are better and an $\varepsilon$-independent asymptotically optimal order of convergence is achieved in all cases. 
  \begin{table}[h!]
    \centering
    \caption{Numerical results for Example \ref{exam6.1} on Shishkin mesh.}
    \label{Exam1}
    \begin{tabular}{|c|c c c c c|}
    \hline
    $\varepsilon$ &$N=8$ &$N=16$& $N=32$   & $N=64$   & $N=128$  \\ \hline
    1e-00 & 1.01e-03 & 2.61e-04 & 6.58e-05 & 1.65e-05 & 4.12e-06 \\
          & 1.96     & 1.99     & 2.00     & 2.00     & --       \\ 
    1e-01 & 3.77e-03 & 1.06e-03 & 2.75e-04 & 6.94e-05 & 1.74e-05 \\
          & 1.83     & 1.95     & 1.99     & 2.00     & --       \\ 
    1e-02 & 1.17e-02 & 6.43e-03 & 3.03e-03 & 1.25e-03 & 4.59e-04 \\
          & 0.86     & 1.09     & 1.28     & 1.44     & --       \\ 
    1e-03 & 3.81e-03 & 2.08e-03 & 9.73e-04 & 4.00e-04 & 1.46e-04 \\
          & 0.87     & 1.10     & 1.28     & 1.45     & --       \\ 
    1e-04 & 1.22e-03 & 6.59e-04 & 3.08e-04 & 1.27e-04 & 4.64e-05 \\
          & 0.89     & 1.10     & 1.28     & 1.45     & --       \\
    1e-05 & 4.18e-04 & 2.09e-04 & 9.75e-05 & 4.01e-05 & 1.47e-05 \\
          & 1.00     & 1.10     & 1.28     & 1.45     & --       \\ 
    1e-06 & 2.09e-04 & 6.71e-05 & 3.09e-05 & 1.27e-05 & 4.64e-06 \\
          & 1.64     & 1.12     & 1.28     & 1.45     & --       \\
    1e-07 & 1.74e-04 & 2.44e-05 & 9.84e-06 & 4.01e-06 & 1.47e-06 \\
          & 2.84     & 1.31     & 1.29     & 1.45     & --       \\ \hline
    \end{tabular}
  \end{table}

  \begin{table}[h!]
    \centering
    \caption{Numerical results for Example \ref{exam6.1} on uniform mesh.}
    \label{Exam1.1}
    \begin{tabular}{|c|c c c c c|}
    \hline
      $\varepsilon$ &$N=8$ &$N=16$& $N=32$   & $N=64$   & $N=128$  \\ \hline
      1e-00 & 1.01e-03 & 2.61e-04 & 6.58e-05 & 1.65e-05 & 4.12e-06 \\
            & 1.96     & 1.99     & 2.00     & 2.00     & --       \\
      1e-01 & 3.77e-03 & 1.06e-03 & 2.75e-04 & 6.94e-05 & 1.74e-05 \\
            & 1.83     & 1.95     & 1.99     & 2.00     & --       \\
      1e-02 & 3.87e-02 & 2.06e-02 & 8.03e-03 & 2.62e-03 & 7.53e-04 \\
            & 0.91     & 1.36     & 1.62     & 1.80     & --       \\
      1e-03 & 1.24e-02 & 1.60e-02 & 1.73e-02 & 1.44e-02 & 8.63e-03 \\
            & -0.37    & -0.12    & 0.26     & 0.74     & --       \\
      1e-04 & 1.30e-03 & 1.83e-03 & 2.57e-03 & 3.56e-03 & 4.74e-03 \\
            & -0.49    & -0.49    & -0.47    & -0.41    & --       \\  
      1e-05 & 1.33e-04 & 1.85e-04 & 2.62e-04 & 3.70e-04 & 5.21e-04 \\
            & -0.47    & -0.50    & -0.50    & -0.49    & --       \\
      1e-06 & 1.96e-05 & 1.87e-05 & 2.62e-05 & 3.71e-05 & 5.25e-05 \\
            & 0.07     & -0.49    & -0.50    & -0.50    & --       \\
      1e-07 & 1.30e-05 & 2.24e-06 & 2.63e-06 & 3.71e-06 & 5.25e-06 \\
            & 2.54     & -0.23    & -0.50    & -0.50    & --       \\ \hline
        \end{tabular}
      \end{table}
  In the case of $k=4$, the numerical results are shown in Table \ref{Exam2}, which has yielded the asymptotically optimal order.
  \begin{table}[h!]
    \centering
    \caption{Numerical results for Example \ref{exam6.1} on Shishkin mesh.}
    \label{Exam2}
    \begin{tabular}{|c|c c c c|}
    \hline
    $\varepsilon$ &$N=8$ &$N=16$& $N=32$   & $N=64$   \\ \hline
    1e-00 & 3.07e-05 & 3.90e-06 & 4.89e-07 & 6.12e-08 \\
          & 2.98     & 3.00     & 3.00     & --       \\
    1e-01 & 3.92e-04 & 5.35e-05 & 6.84e-06 & 8.61e-07 \\
          & 2.87     & 2.97     & 2.99     & --       \\ 
    1e-02 & 6.08e-03 & 2.56e-03 & 8.25e-04 & 2.11e-04 \\
          & 1.25     & 1.63     & 1.97     & --       \\ 
    1e-03 & 1.98e-03 & 8.29e-04 & 2.66e-04 & 6.77e-05 \\
          & 1.26     & 1.64     & 1.97     & --       \\ 
    1e-04 & 6.28e-04 & 2.63e-04 & 8.43e-05 & 2.15e-05 \\
          & 1.26     & 1.64     & 1.97     & --       \\
    1e-05 & 1.99e-04 & 8.32e-05 & 2.67e-05 & 6.79e-06 \\
          & 1.26     & 1.64     & 1.97     & --       \\
    1e-06 & 6.33e-05 & 2.63e-05 & 8.44e-06 & 2.24e-06 \\
          & 1.27     & 1.64     & 1.91     & --       \\ \hline
    \end{tabular}
    \end{table}

\begin{example}\label{exam6.2}
  Let $g(x)$ be as in Example \ref{exam6.1} and set
  $$p(y)=2y(1-y^{2})+\varepsilon\left[ld(1-2y)-3\frac{q}{l}+\left(\frac{3}{l}-d\right)e^{-y/\varepsilon}+\left(\frac{3}{l}+d\right)e^{(y-1)/ \varepsilon}\right],$$
  with $l = 1-e^{-1/ \varepsilon}$, $q=2-l$ and $d=1/(q-2\varepsilon l)$. Then choose $f(x, y)$ such that the exact solution of (\ref{1.1}) is $u(x, y)=g(x)p(y)$. 
  The numerical results are shown in Table \ref{Exam3}-\ref{Exam4}.
\end{example}

The error and the order of convergence on Shishkin mesh and uniform mesh with $k=3$ are listed in Table \ref{Exam3} and Table \ref{Exam3.1}, respectively. And also we get better results on Shishkin mesh than uniform mesh.
  \begin{table}[h!]
    \centering
    \caption{Numerical results for Example \ref{exam6.2} on Shishkin mesh.}
    \label{Exam3}
    \begin{tabular}{|c|c c c c c|}
    \hline
    $\varepsilon$ &$N=8$ &$N=16$& $N=32$   & $N=64$   & $N=128$  \\ \hline
    1e-00 & 1.66e-04 & 4.24e-05 & 1.07e-05 & 2.67e-06 & 6.67e-07 \\
          & 1.97     & 1.99     & 2.00     & 2.00     & --       \\ 
    1e-01 & 6.48e-03 & 1.82e-03 & 4.72e-04 & 1.19e-04 & 2.99e-05 \\
          & 1.83     & 1.94     & 1.99     & 2.00     & --       \\ 
    1e-02 & 2.10e-02 & 1.16e-02 & 5.44e-03 & 2.24e-03 & 8.25e-04 \\
          & 0.86     & 1.09     & 1.28     & 1.44     & --       \\ 
    1e-03 & 6.86e-03 & 3.74e-03 & 1.75e-03 & 7.20e-04 & 2.64e-04 \\
          & 0.87     & 1.10     & 1.28     & 1.45     & --       \\ 
    1e-04 & 2.18e-03 & 1.19e-03 & 5.55e-04 & 2.28e-04 & 8.35e-05 \\
          & 0.88     & 1.10     & 1.28     & 1.45     & --       \\
    1e-05 & 7.13e-04 & 3.76e-04 & 1.76e-04 & 7.22e-05 & 2.64e-05 \\
          & 0.92     & 1.10     & 1.28     & 1.45     & --       \\ 
    1e-06 & 2.87e-04 & 1.20e-04 & 5.56e-05 & 2.28e-05 & 8.36e-06 \\
          & 1.27     & 1.11     & 1.28     & 1.45     & --       \\
    1e-07 & 2.00e-04 & 4.01e-05 & 1.76e-05 & 7.29e-06 & 2.64e-06 \\
          & 2.32     & 1.19     & 1.27     & 1.46     & --       \\ \hline
    \end{tabular}
    \end{table}

    \begin{table}[h!]
      \centering
      \caption{Numerical results for Example \ref{exam6.2} on uniform mesh.}
      \label{Exam3.1}
      \begin{tabular}{|c|c c c c c|}
      \hline
        $\varepsilon$ &$N=8$ &$N=16$& $N=32$   & $N=64$   & $N=128$  \\ \hline
        1e-00 & 1.66E-04 & 4.24E-05 & 1.07E-05 & 2.67E-06 & 6.67E-07 \\
              & 1.97     & 1.99     & 2.00     & 2.00     & --       \\
        1e-01 & 6.48E-03 & 1.82E-03 & 4.72E-04 & 1.19E-04 & 2.99E-05 \\
              & 1.83     & 1.94     & 1.99     & 2.00     & --       \\
        1e-02 & 6.96E-02 & 3.70E-02 & 1.44E-02 & 4.71E-03 & 1.35E-03 \\
              & 0.91     & 1.36     & 1.62     & 1.80     & --       \\
        1e-03 & 2.22E-02 & 2.88E-02 & 3.12E-02 & 2.60E-02 & 1.56E-02 \\
              & -0.37    & -0.12    & 0.26     & 0.74     & --       \\
        1e-04 & 2.35E-03 & 3.30E-03 & 4.63E-03 & 6.41E-03 & 8.54E-03 \\
              & -0.49    & -0.49    & -0.47    & -0.41    & --       \\  
        1e-05 & 2.37E-04 & 3.34E-04 & 4.72E-04 & 6.66E-04 & 9.38E-04 \\
              & -0.49    & -0.50    & -0.50    & -0.49    & --       \\
        1e-06 & 2.85E-05 & 3.35E-05 & 4.73E-05 & 6.69E-05 & 9.45E-05 \\
              & -0.23    & -0.50    & -0.50    & -0.50    & --       \\
        1e-07 & 1.45E-05 & 3.62E-06 & 4.73E-06 & 6.69E-06 & 9.46E-06 \\
              & 2.00     & -0.39    & -0.50    & -0.50    & --       \\ \hline
        \end{tabular}
      \end{table}

    In the case of $k=4$, the numerical results are shown in Table \ref{Exam4}, which has also yielded the asymptotically optimal order, independent of $\varepsilon$.
    \begin{table}[h!]
      \centering
      \caption{Numerical results for Example \ref{exam6.2} on Shishkin mesh.}
      \label{Exam4}
      \begin{tabular}{|c|c c c c|}
      \hline
      $\varepsilon$ &$N=8$ &$N=16$& $N=32$   & $N=64$   \\ \hline
      1e-00 & 3.84e-06 & 4.86e-07 & 6.09e-08 & 7.62e-09 \\
            & 2.98     & 3.00     & 3.00     & --       \\
      1e-01 & 6.93e-04 & 9.45e-05 & 1.21e-05 & 1.52e-06 \\
            & 2.87     & 2.97     & 2.99     & --       \\ 
      1e-02 & 1.09e-02 & 4.60e-03 & 1.48e-03 & 3.79e-04 \\
            & 1.25     & 1.63     & 1.97     & --       \\ 
      1e-03 & 3.57e-03 & 1.49e-03 & 4.79e-04 & 1.22e-04 \\
            & 1.26     & 1.64     & 1.97     & --       \\
      1e-04 & 1.13e-03 & 4.74e-04 & 1.52e-04 & 3.87e-05 \\
            & 1.26     & 1.64     & 1.97     & --       \\
      1e-05 & 3.58e-04 & 1.50e-04 & 4.81e-05 & 1.22e-05 \\
            & 1.26     & 1.64     & 1.97     & --       \\
      1e-06 & 1.14e-04 & 4.74e-05 & 2.53e-05 & 3.90e-06 \\
            & 1.26     & 0.90     & 2.70     & --       \\ \hline
      \end{tabular}
      \end{table}

\begin{remark}
	The numerical examples show that we obtain the asymptotically optimal error estimate. In fact, the $\ln ^{k-\frac{3}{2}}N$ in the error results will have an impact on the convergence order, and with the refinement of the Shishkin mesh, the influence of $\ln ^{k-\frac{3}{2}}N$ on the convergence order will gradually become smaller. The convergence order is reduced by about half order when $k=3$ and by one order in the case of $k=4$. Our numerical results also confirm this well.
\end{remark}

\section{Conclusion}
In this paper, we use the weak Galerkin finite element method to solve the singularly perturbed fourth-order boundary value problem in 2D domain. By constructing the Shishkin mesh which is suitable for the problem, we give the corresponding numerical algorithm and the error estimate in the $H^2$ discrete norm. Compared with solving the problem on the uniform mesh, the WG method obtains better results on the Shishkin mesh, which is verified by the numerical results. Moreover, the results of our numerical experiments are consistent with the error estimation theory, and the asymptotically optimal convergence order is obtained.

\section*{Statements and Declarations}

\smallskip
\noindent
\textbf{Funding}.
This work was supported by the National Natural Science Foundation of China (Grant No. 12101039, 12271208).

\smallskip
\noindent
\textbf{Data Availability}.
The code used in this work will be made available upon request to the authors.

\smallskip
\noindent
\textbf{Competing Interests}.
The authors have no relevant financial or non-financial interests to disclose.


\newpage

\end{document}